\documentclass{amsart}
\usepackage{mathrsfs}
\usepackage{stmaryrd}
\usepackage[alphabetic]{amsrefs}
\usepackage[all]{xy}

\title[Overconvergent unit-root $F$-isocrystals and isotriviality]
{Overconvergent unit-root $F$-isocrystals and isotriviality}

\author{Teruhisa Koshikawa}
\address{Department of Mathematics, University of Chicago}
\email{teruhisa@uchicago.edu}

\newcommand\bF{\mathbf F}
\newcommand\bQ{\mathbf Q}
\newcommand\bZ{\mathbf Z}

\theoremstyle{plain}
\newtheorem{thm}{Theorem}[section]
\newtheorem{lem}[thm]{Lemma}
\newtheorem{prop}[thm]{Proposition}

\theoremstyle{definition}
\newtheorem{rem}[thm]{Remark}
\newtheorem{defn}[thm]{Definition}
\newtheorem{exam}[thm]{Example}

\begin{document}

\begin{abstract}
We show that a semisimple overconvergent ``absolutely unit-root" $F$-isocrystal on a geometrically connected smooth variety over a finite field becomes constant over a finite covering. 
\end{abstract}

\maketitle

\section{Introduction}
We will study overconvergent unit-root $F$-isocrystals on a smooth variety over a finite field. 
We first state a result for projective smooth curves, using $p$-adic representations of the \'etale fundamental group. 

Let $X$ be a geometrically connected projective smooth curve over a finite field $\bF_q$ of characteristic $p$. 
We will write $\overline{\bF}_q$ for an algebraic closure of $\bF_q$, and $G_{\bF_q}$ for the absolute Galois group of $\bF_q$.  
Then, we have the following exact sequence of fundamental groups:
\[
1\to \pi_1(X_{\overline{\bF}_q})\to \pi_1(X)\to G_{\bF_q}\to 1. 
\]
It is a well-known consequence of class field theory that the image of the geometric fundamental group $\pi_1(X_{\overline{\bF}_q})$ in the abelianization $\pi_1(X)^{\textnormal{ab}}$ is finite. 

We consider a $p$-adic non-abelian version by using the Langlands program (\cite{Lafforgue} and~\cite{Abe}). For this, we need to introduce a condition on Frobenius eigenvalues. Let $\rho$ be a continuous finite dimensional $\overline{\bQ}_p$-representation of $\pi_1(X)$. For any closed point $x$ of $X$, we shall consider Forbenius eigenvalues of geometric Frobenius at $x$. If $\rho$ is irreducible with finite determinant, $\rho$ is absolutely unit-root if the following holds; for any Frobenius eigenvalue $\lambda$, which is known to be an algebraic number, and any ring homomorphism $\iota\colon\overline{\bQ}_p\to\overline{\bQ}_p$, $\iota(\lambda)$ is a $p$-adic unit. We say that $\rho$ is absolutely unit-root if any irreducible constitute $\rho'$ of $\rho$ has an absolutely unit-root twist $\rho'\otimes\chi$ for some character $\chi$. 

\begin{prop}\label{fund}
Let $X$ be a geometrically connected projective smooth curve over $\bF_q$ and $\rho$ a continuous finite dimensional $\overline{\bQ}_p$-representation of $\pi_1(X)$. 
Suppose $\rho$ is absolutely unit-root. 

If $\rho$ is either semisimple or any irreducible constitute of $\rho$ has the determinant of finite order, then the image of $\pi_1(X_{\overline{\bF}_q})$ under $\rho$ is finite. 
In other words, $\rho$ factors through $G_{\bF_q}$ after replacing $X$ and $\bF_q$ by finite \'etale coverings. 

If $\rho$ is semisimple and any irreducible constitute of $\rho\otimes\overline{\bQ}_p$ has the determinant of finite order, the whole action of $\pi_1(X)$ factors through a finite quotient. 
\end{prop}

This is an analogue of the unramified Fontaine-Mazur conjecture in the number field case (\cite{Fontaine-Mazur}*{Conjecture 5a} and~\cite{Kisin-Wortmann}). Some assumption for a non-semisimple $\rho$ is necessary, see Example~\ref{counterexample}. 

Proposition~\ref{fund} can be generalized to smooth varieties. 
We prefer to state such a generalization in terms of overconvergent $F$-isocrystals as follows. 

We let $X$ be a smooth variety over $\bF_q$. 
As we recall in Section~\ref{pre}, a $p$-adic representation of $\pi_1(X)$ corresponds to a unit-root (convergent) $F$-isocrystal on $X$. 
For a non-proper $X$, one should put a condition at infinity to generalize Proposition~\ref{fund}, and a suitable notion is an overconvergent unit-root $F$-isocrystal on $X$. 
Let $F\textnormal{-Overconv}(X/{\overline{\bQ}_p})^{\circ}$ denote the category of overconvergent unit-root $F$-isocrystals on $X$ with coefficients in $\overline{\bQ}_p$. 

We define a notion of overconvergent absolutely unit-root $F$-isocrystal in the same way as absolutely unit-root representations. See Definition~\ref{abs} for details. If $M$ is the $p$-adic realization of a ``pure motive" over $X$ with coefficients in $\bQ$, unit-root would imply absolutely unit-root but this is not the case in general, see Example~\ref{Drinfeld}. 

We say that an $F$-isocrystal is constant if it is the pullback of an $F$-isocrystal on a finite field. 
Then, our goal in this paper is to prove the following theorem: 

\begin{thm}\label{isotrivial}
Let $X$ be a geometrically connected smooth variety over $\bF_q$ and $M$ an object of $F\textnormal{-Overconv}(X/{\overline{\bQ}_p})^{\circ}$. 
Suppose $M$ is absolutely unit-root. 

If $M$ is either semisimple as an overconvergent $F$-isocrystal or any irreducible constitute of $M$ has the determinant of finite order, then $M$ is isotrivial, i.e., becomes constant over a finite \'etale covering of $X$. 

If $M$ is semisimple and any irreducible constitute of $M$ has the determinant of finite order, $M$ becomes trivial after a finite \'etale covering.
\end{thm} 

An $\ell$-adic version of the theorem is discussed in Remark~\ref{l-adic} below. 
Another related $\ell$-adic statement is the de Jong conjecture~\cite{deJong}. 

If $X$ is projective, (a $p$-adic representation version of) Theorem~\ref{isotrivial} follows from Proposition~\ref{fund} by considering hyperplane sections. 
The general case can be reduced to the projective case as a $p$-adic representation corresponding to $M$ is known to be potentially unramified, see Section~\ref{pre}. 

After a preparation in Section~\ref{pre}, we prove Prposition~\ref{fund} in Section~\ref{proof of fund}, and Theorem~\ref{isotrivial} in Section~\ref{proof of isotrivial}. 

\subsection*{Acknowledgments}
I am grateful to Keerthi Madapusi Pera for helpful discussions. 
I would like to thank Kazuya Kato for discussions, suggestions and comments. 
Vladimir Drinfeld kindly pointed out a gap in an earlier version and suggested Example~\ref{Drinfeld}, and I would like to thank him.  
I also thank Koji Shimizu and H\'el\`ene Esnault for their comments. 

\subsection*{Notation}
Throughout this note, $\pi_1$ denotes \'etale fundamental groups. 
We do not specify fixed geometric points.  

\section{Preliminaries}\label{pre}
We will review relations between $F$-isocrystals and $p$-adic representations of fundamental groups. 
In particular, we explain why Theorem~\ref{isotrivial} generalizes Proposition~\ref{fund}. 

Let $X$ be a smooth variety over $\bF_q$. 
We first review a unit-root $F$-crystal $M$ on $X$ relative to $W$. It is a sheaf on the crystalline site $(X/W)_{\textnormal{crys}}$ whose values are locally free of finite rank, and it is equipped with an isomorphism $\varphi^* M\overset{\cong}{\longrightarrow}M$. Here $\varphi$ denotes a power of the absolute Frobenius of $X$ defined by the $q$-th power. 
% Abe's convention has a problem with local-global compatibility?
We write $F\textnormal{-Crys}(X/W)^{\circ}$ for the category of unit-root $F$-crystals. 

Recall also that there is a notion of unit-root convergent $F$-isocrystal (relative to $K=W(\bF_q)[\frac{1}{p}]$) as discussed in~\cite{Crew}*{Section 1}. It can be regarded as a sheaf $M$ on the category of enlargements whose values are locally free of finite rank, and equipped with an isomorphism $\varphi^* M\overset{\cong}{\longrightarrow}M$, which is fiberwise unit-root as an $F$-isocrystal on a point. 
We write $F\textnormal{-Conv}(X/K)^{\circ}$ for the category of unit-root convergent $F$-isocrystals. 
By~\cite{Crew}*{2.1. Theorem and 2.2. Theorem}, there is an equivalence of categories
\[
F\textnormal{-Crys}(X/W)^{\circ}\otimes_{\bZ}\bQ \overset{\cong}{\longrightarrow}
F\textnormal{-Conv}(X/K)^{\circ}. 
\]

Further recall a notion of unit-root overconvergent(``=overconvergent unit-root" as below) $F$-isocrystals (relative to $K$) as defined in~\cite{Crew}*{Section 1}, for instance. We will not recall the definition, but we mention that there is a natural (sort of forgetful) functor
\[
F\textnormal{-Overconv}(X/K)^{\circ}\longrightarrow
F\textnormal{-Conv}(X/K)^{\circ}, 
\]
where $F\textnormal{-Overconv}(X/K)^{\circ}$ denotes the category of unit-root overconvergent $F$-isocrystals. 
This functor is known to be fully faithful\footnote{This is the precise meaning of ``overconvergent unit-root=unit-root overconvergent".} (Tsuzuki~\cite{Tsuzuki:hom}*{Corollary 1.2.3} and~\cite{Kedlaya:semistableII}*{Theorem 4.2.1}). 
Moreover, this is an equivalence if $X$ is proper. 

These categories of $F$-isocrystals we defined are $K$-linear categories, and we may extend the scalars to a fixed algebraic closure $\overline{\bQ}_p$ of $\bQ_p$, cf.~\cite{Abe}. The category of unit-root overconvergent (resp. convergent) $F$-isocrystals with $\overline{\bQ}_p$-structure will be denoted by $F\textnormal{-Overconv}(X/\overline{\bQ}_p)^{\circ}$ (resp. $F\textnormal{-Conv}(X/\overline{\bQ}_p)^{\circ}$). 

We shall explain the relations to representations of $\pi_1(X)$. Let $\textnormal{Rep}_{W(\bF_q)}(\pi_1(X))$ denote the category of continuous finite free $W(\bF_q)$-representations of $\pi_1(X)$ and $\textnormal{Rep}_{\overline{\bQ}_p}(\pi_1(X))$ the category of continuous finite dimensional $\overline{\bQ}_p$-representations of $\pi_1(X)$. 
There are two compatible equivalences of categories
\[
F\textnormal{-Crys}(X/W)^{\circ}\overset{\cong}{\longrightarrow}\textnormal{Rep}_{W(\bF_q)}(\pi_1(X)), 
F\textnormal{-Conv}(X/\overline{\bQ}_p)^{\circ}\overset{\cong}{\longrightarrow}\textnormal{Rep}_{\overline{\bQ}_p}(\pi_1(X)). 
\]
The former is \cite{Katz}*{Proposition 4.1.1}, and the latter is~\cite{Crew}*{2.2. Theorem}. 

Moreover, if $M$ is an object of $F\textnormal{-Overconv}(X/\overline{\bQ}_p)^{\circ}$, the corresponding $\overline{\bQ}_p$-representation $\rho$ is potentially unramified in the sense of \cite{Kedlaya:unit-root}*{Definition 2.3.5}. 
This important fact was proved by Tsuzuki~\cite{Tsuzuki:hom}*{Proposition 7.2.1}. 

In fact, any potentially unramified representation corresponds to a unit-root overconvergent $F$-isocrystal and it induces an equivalence of such categories as shown by Tsuzuki~\cite{Tsuzuki:curve}*{Theorem 7.2.3} in the curve case\footnote{Technically speaking, one should also use~\cite{Shiho}*{Proposition 4.2}.} and by Kedlaya~\cite{Kedlaya:unit-root}*{Theorem 2.3.7} in the general case. (See also~\cite{Abe:cycle}*{Proposition 1.4.3} with a general theory of six operations~\cite{Abe}.)

We would like to mention a result of Shiho~\cite{Shiho}*{Theorem 4.3}. He showed that a representation that has finite local monodromy along the generic points of  ``boundary divisors" is potentially unramified, see {\it ibid.} for the precise definition. We will not use this result in this article. 

We will introduce a condition on Frobenius eigenvalues. For any closed point $x\in X$, using $\rho$, we shall consider the eigenvalues and the characteristic polynomial of the geometric Frobenius at $x$. These Frobenius eigenvalues can be identified (up to normalization) with those of the overconvergent $F$-isocrystal $M$. 

\begin{defn}\label{abs}
Let $M$ be an object of $F\textnormal{-Overconv}(X/\overline{\bQ}_p)^{\circ}$. 
\begin{enumerate}
\item Suppose $M$ is irreducible with finite determinant. Any Frobenius eigenvalue is known to be an algebraic number by~\cite{Abe}. We say that $M$ is absolutely unit-root if $\iota(\lambda)$ is a $p$-adic unit for any Frobenius eigenvalue $\lambda$ and any ring homomorphism $\iota\colon\overline{\bQ}_p\to\overline{\bQ}_p$. 
\item Suppose $M$ is irreducible. It is isomorphic to a twist of irreducible $M'$ with finite determinant~\cite{Abe2011}*{6.1 Lemma.(ii)}. We say that $M$ is absolutely unit-root if $M'$ is absolutely unit-root. This is independent of twisting $M'$. 
\item We say that $M$ (or corresponding $\rho$) is absolutely unit-root if every irreducible constitute of $M$ is absolutely unit-root.  
\end{enumerate}
\end{defn}

Note that any Frobenius eigenvalue $\lambda$ is a $p$-adic unit for a unit-root $F$-isocrystal. The following example suggested by Drinfeld gives a unit-root $F$-isocrystal that is not absolutely unit-root. 

\begin{exam}\label{Drinfeld}
We consider Hilbert-Blumenthal modular varieties and their stratifications as in \cite{Goren}. 
Let $L$ be a totally real field of degree $d>1$ over $\bQ$ with the ring $\mathcal{O}_L$ of integers, and suppose $p$ splits completely in $L$ and is prime to the discriminant of $L$.  
For an integer $n\geq 3$ prime to $p$, let $\mathcal{M}^{n}_{L}$ be a fine moduli scheme over $\bF_p$ of Hilbert-Blumenthal abelian schemes with real multiplication by $\mathcal{O}_L$ and full level $n$ structure. (There is a condition on polarizations, but we will not recall it here.) It is smooth of dimension $d$. 

We denote by $\mathbb{D}$ the overconvergent $F$-isocrystal (or contravariant crystalline Dieudonn\'e module) of the universal abelian scheme over $\mathcal{M}^{n}_{L}$.  
Given a decomposition $\mathcal{O}_{L}\otimes\bZ_p=\prod_v \mathcal{O}_{L,v}$, $\mathbb{D}$ has a decomposition $\mathbb{D}=\bigoplus_v \mathbb{D}_v$ such that $\mathbb{D}_v$ has $\mathcal{O}_{L,v}$-structure of rank $2$. 

Fix a place $v_0$ of $L$ dividing $p$. Goren \cite{Goren}*{Section 1} defined a smooth divisor $Y$ of $\mathcal{M}^{n}_{L}$ such that, at the generic point of $Y$, $\mathbb{D}_{v_0}$ is supersingular and $\mathbb{D}_v$ is ordinary for every $v\neq v_0$. In particular, $\mathbb{D}_{v_0}$ is supersingular over $Y$. 

Some twist of $\mathbb{D}_{v_0}$ is unit-root over $Y$, but not absolutely unit-root. This is because $\mathbb{D}_v$ with $v\neq v_0$ is a crystalline companion of $\mathbb{D}_{v_0}$, i.e., they share Frobenius eigenvalues, and $\mathbb{D}_v$ is not supersingular over $Y$.  
\end{exam}

Finally, we shall explain the relation between Proposition~\ref{fund} and Theorem~\ref{isotrivial}. Suppose $M$ satisfies one of the assumptions in Theorem~\ref{isotrivial}. 
Since the category of potentially unramified representations is closed under subobjects and quotients, we see that $\rho$ satisfies a similar assumption for $M$. 
Therefore, Theorem~\ref{isotrivial} for $M$ is equivalent to the corresponding statement for $\rho$, hence generalizes Proposition~\ref{fund}. 

\section{The proof of Proposition~\ref{fund}}\label{proof of fund}

In this section, $X$ is a geometrically connected projective smooth curve over $\bF_q$. 

\begin{lem}\label{extension}
If Proposition~\ref{fund} holds for any semisimple absolutely unit-root $\rho$, it holds for any $\rho$ satisfying the assumption of Proposition~\ref{fund}. 
\end{lem}

\begin{proof}
Take an indecomposable $\overline{\bQ}_p$-representation $\rho$ of $\pi_1(X)$ as in the statement of Proposition~\ref{fund}. 
By the assumption, any irreducible constitute becomes trivial as a representation of $\pi_1(X)$ after a finite covering of $X$, and we may assume all irreducible constitutes are trivial. 

We will show that the action of $\pi_1(X)$ on such a $\rho$ factors through $G_{\bF_q}$. By induction, it suffices to show this statement when $\dim\rho=2$. 
Then, $\rho$ is an extension of the trivial character by itself, and is classified by $H^1(\pi_1(X), \overline{\bQ}_p)=\textnormal{Hom}(\pi_1(X), \overline{\bQ}_p)$.  By the class field theory, any element in $\textnormal{Hom}(\pi_1(X), \overline{\bQ}_p)$ factors through $G_{\bF_q}$. This means it comes from $H^1(G_{\bF_q},\overline{\bQ}_p)$ and the corresponding extension comes from a representation of $G_{\bF_q}$.  
\end{proof}

It suffices to show Proposition~\ref{fund} for an irreducible absolutely unit-root $\overline{\bQ}_p$-representation $\rho$. 
As in Section~\ref{pre}, $\rho$ corresponds to an absolutely unit-root overconvergent $F$-isocrystal $M$ with $\overline{\bQ}_p$-structure. 
Since this correspondence is the equivalence of categories, irreducibility is preserved. 
In fact, $M$ is irreducible in the category $F\textnormal{-Overconv}(X/\overline{\bQ}_p)$ of (not necessary unit-root) overconvergent $F$-isocrystals with $\overline{\bQ}_p$-structure as unit-root condition is closed under subquotients. 

The first key fact, which is a $p$-adic version of Delinge's conjecture~\cite{Deligne:WeilII}*{Conjecture 1.2.10 (ii)}, is a consequence of the $p$-adic version of the Langlands correspondence established by T. Abe~\cite{Abe} following Drinfeld-Lafforgue's strategy. 

\begin{thm}[\cite{Abe}]\label{finite}
For any irreducible object of $F\textnormal{-Overconv}(X/\overline{\bQ}_p)$ with finite determinant, there exists a finite extension $E$ of $\bQ$ inside $\overline{\bQ}_p$ such that, for any closed point $x\in X$, the characteristic polynomial of Frobenius at $x$ has coefficients in $E$. 
\end{thm}

This follows from the Langlands correspondence and the corresponding statement for cuspidal automorphic representations, see also the proof of~\cite{Lafforgue}*{Th\'eor\`eme VII.6.(i)}.

Our $M$ may not satisfy the assumption of Theorem~\ref{finite}, but after twisting~\cite{Abe2011}*{6.1 Lemma.(ii)} we may assume that the determinant of $M$ is of finite order, and we can apply Theorem~\ref{finite}. 

The second ingredient is: 

\begin{prop}\label{root of unity}
Let $\rho$ be an irreducible object $\rho$ of $\textnormal{Rep}_{\overline{\bQ}_p}(\pi_1(X))$ with finite determinant. If $\rho$ is absolutely unit-root, any Frobenius eigenvalue is a root of unity. 
\end{prop}

\begin{proof}
As before, we write $M$ for the corresponding overconvergent unit-root $F$-isocrystal. 

Given a Frobenius eigenvalue of $M$, take a finite extension $E'$ of $E$ containing it. 
By~\cite{Abe} and ~\cite{Lafforgue}*{Th\'eor\`eme VII.6.(ii)-(iii)}, any Frobenius eigenvalue has the complex absolute values $1$, and is a $v$-adic unit for any finite place $v$ of $E'$ prime to $p$.  
(Roughly speaking, Lafforgue constructed a pure motive corresponding to $\rho$. Given such a pure motive, these results on Frobenius eigenvalues follow from the theory of weights~\cite{Deligne:Weil II} or its $p$-adic counterpart (\cite{Kedlaya:Weil} and~\cite{Abe-Caro}).)

Since $M$ is absolutely unit-root, it is also a $v$-adic unit for any $v$ above $p$. 
Such a number must be a root of unity. 
\end{proof}

Finally, we explain how to combine Theorem~\ref{finite} and Proposition~\ref{root of unity} to prove Proposition~\ref{fund}. 
Since the degree of the characteristic polynomial of Frobenius is fixed and independent of $x$, Theorem~\ref{finite} and Proposition~\ref{root of unity} imply that there are only finitely many possibilities for the characteristic polynomial of Frobenius. 
Since the action is continuous, using the Chebotarev density, we see that there are also only finitely many possibilities for the characteristic polynomial of any element in $\pi_1(X)$. 
Therefore, again by continuity, there is an open normal subgroup $H$ of $\pi_1(X)$ acting unipotently.  
Since $\rho$ is irreducible and $H$ is normal, its restriction to $H$ is semisimple, hence trivial. 

\begin{rem}
If the number field $E$ is equal to $\bQ$, absolutely unit-root assumption is automatically satisfied.  
\end{rem}

\begin{exam}\label{counterexample}
For a general absolutely unit-root $\rho$ not satisfying the assumption of Proposition~\ref{fund}, Proposition~\ref{fund} is false. 
For example, let $T_p=\varprojlim J(\overline{\bF_q})[p^n]$ be the maximal \'etale quotient of $p$-adic Tate module of the Jacobian $J$ of $X$. It is a representation of $G_{\bF_q}$ and known to be semisimple after inverting $p$. Via the surjection $\pi_1(X)\to G_{\bF_q}$, we regard it as a representation of $\pi_1(X)$, which is absolutely unit-root. 
Assuming $J$ is not supersingular, we will make a counterexample as an extension of the trivial character by $T_p\otimes\bQ$, i.e., its geometric monodromy is not finite. 

Via the Kummer sequence, $H^1(G_{\bF_q}, T_p)$ is isomorphic to the $p$-primary part of $J(\bF_q)$. 

On the other hand, we have the following map
\[
J(X)=\textnormal{Mor}_{\bF_q}(X, J)\to H^1_{\textnormal{flat}}(X, J[p^n])\to H^1(\pi_1(X), J(\overline{\bF}_q)[p^n]). 
\]
By taking the limit, we obtain $J(X)\to H^1(\pi_1(X), T_p)$. 

Suppose $q$ is large enough so that $X$ has an $\bF_q$-rational point $x_0$. 
It induces a nonconstant map $f\colon X\to J$ sending $x_0$ to $0$.  
We prove that the image of $f$ in $H^1(\pi_1(X), T_p\otimes\bQ)$ is nonzero and gives a desired counterexample. 

Since the action of $\pi_1(X_{\overline{\bF}_q})$ is trivial on $T_p$, 
\[
H^1(\pi_1(X_{\overline{\bF}_q}), T_p)=\textnormal{Hom}(\pi_1(X_{\overline{\bF}_q})^{\textnormal{ab}}, T_p)=\textnormal{Hom}(T_p, T_p). 
\]
Then the image of $f$ in $\textnormal{Hom}(T_p, T_p)$ can be identified with the identity. 
This implies the desired claim. 
\end{exam}

\section{The proof of Theorem~\ref{isotrivial}}\label{proof of isotrivial}
As explained in Section~\ref{pre}, let $\rho$ be a potentially unramified $\overline{\bQ}_p$-representation of $\pi_1(X)$ corresponding to $M$.  
We will prove the corresponding statement for $\rho$.  

Note that we can shrink $X$ if necessary. 
Then, by de Jong's alteration~\cite{deJong:alteration}*{4.1 Theorem}, we can take a connected finite \'etale covering $Y$ of $X$ with an open dense immersion into a projective smooth variety $\overline{Y}$. 
Since $\rho$ is potentially unramified~\cite{Kedlaya:unit-root}*{Theorem 2.3.7}, after passing to another connected finite \'etale covering if necessary, $\rho$ factors through $\pi_1(\overline{Y})$. 

First, we treat the case $M$ is semisimple. So, $\rho$ is also semisimple. As semisimplicity is preserved by finite \'etale coverings, $\rho$ is also semisimple as a representation of $\pi_1(\overline{Y})$. 
By the Bertini theorem and the Lefschetz-type theorem for $\pi_1$~\cite{SGA1}*{Expos\'e X, Lemma 2.10}, after a finite base extension\footnote{In fact, this is not necessarily by ~\cite{Poonen}*{Theorem 1.1}.}, one can find a geometrically connected projective smooth curve $C$ inside $\overline{Y}$ inducing a surjection $\pi_1(C)\to\pi_1(\overline{Y})$. 
So, the isotriviality of $M$ follows from the curve case, i.e., Proposition~\ref{fund}. 

Next, assume $M$ is irreducible and its determinant is of finite order. 
Then, for some connected finite Galois covering $Z$ of $X$,  $\pi_1(Z_{\overline{\bF}_q})$ acts trivially on $\rho$. In particular, $\rho$ is the direct sum of characters of $\pi_1(Z)$. 
Let $\chi$ be such a character of $\pi_1(Z)$. Note that $\chi$ is also trivial on $\pi_1(Z_{\overline{\bF}_q})$. 
By a usual Clifford description of the restriction of an irreducible representation to a normal subgroup of finite index, such characters are related to each other by conjugation of elements of $\pi_1(X)$. 
Namely, any such character has the form of $g\mapsto \chi(hgh^{-1})$ with $g\in \pi_1(Z)$ and $h \in \pi_1(X)$. 
In fact, as $\chi$ is trivial on $\pi_1(Z_{\overline{\bF}_q})$, this implies that the all these characters are equal to $\chi$. 
Since the product of all these characters is of finite order by the assumption, $\chi$ is finite. 
Therefore, the action of $\pi_1(X)$ on $\rho$ factors through a finite quotient. 

Finally, assume that the determinant of any irreducible constitute of $M$ is of finite order. 
We may assume that any irreducible constitute is trivial. 
As in the beginning of the proof, we can find a finite \'etale covering $Y$ of $X$ with a smooth compactification $\overline{Y}$ such that $\rho$ becomes a representation of $\pi_1(\overline{Y})$. So, we can assume $X$ is a geometrically connected projective smooth variety. 
Then, by~\cite{Katz-Lang}*{Theorem 2}, the kernel of $\pi_1(X)^{\textnormal{ab}}\to G_{\bF_q}^{\textnormal{ab}}$ is finite. 
Using this finiteness, one can see that the action of $\pi_1(X)$ on $\rho$ factors through $G_{\bF_q}$ as in the proof of Lemma~\ref{extension}. 

\begin{rem}\label{l-adic}
Let $\ell$ be a prime different from $p$ and $\mathcal{F}$ an irreducible $\overline{\bQ}_{\ell}$-local system on a (connected) normal scheme of finite type over $\bF_q$ whose determinant is of finite order. 
Then, any Frobenius eigenvalue is an algebraic number (\cite{Lafforgue}*{Th\'eor\`em VII.7.(i)} and~\cite{Deligne:finite}*{1.5-9}). 

Our proof of the theorem can be modified to show that, if all the Frobenius eigenvalues of $\mathcal{F}$ are $p$-adic units, then $\mathcal{F}$ becomes trivial after passing to a finite \'etale covering. 
In the smooth curve case, the same argument works even when it is not proper. 
The reduction to the curve case is done by~\cite{Deligne:finite}*{1.7-8}. (See also~\cite{Esnault-Kerz}*{Appendix B}.)

This also shows the existence of crystalline companions~\cite{Abe} in this particular case. 
Note that it is not known in full generality if $\dim X>1$, though one expects a $p$-adic version of~\cite{Drinfeld} if $X$ is smooth.
\end{rem}

\end{document}